\documentclass[12pt, reqno]{amsart}
\usepackage{amssymb, amsmath, amsthm, hyperref, graphicx}

\newtheorem{thrm}{Theorem}
\newtheorem{lemma}{Lemma}

\begin{document}
\title[The global rigidity of a framework is not affine-invariant]{The global rigidity 
of a framework\\ is not an affine-invariant property}
\author{Victor Alexandrov}
\address{Sobolev Institute of Mathematics, Koptyug ave., 4, 
Novosibirsk, 630090, Russia and Department of Physics, 
Novosibirsk State University, Pirogov str., 2, Novosibirsk, 
630090, Russia}
\email{alex@math.nsc.ru}
\address{{}\hfill{March 9, 2019}}
\begin{abstract}
It is well-known that the property of a bar-and-joint framework `to be
infinitesimally rigid' is invariant under projective transformations of Eucliean $d$-space
for every $d\geqslant 2$.
It is less known that the property of a bar-and-joint framework `to be
globally rigid' is not invariant even under affine transformations of the Euclidean plane.
In this note, we prove of the latter statement for Euclidean $d$-space for every $d\geqslant 2$.

\end{abstract}
\maketitle

\section{Introduction}\label{s1}

A bar-and-joint framework (or a straight-line realization of a graph) 
in $\mathbb{R}^d$, $d\geqslant 2$, is said to be globally rigid in 
$\mathbb{R}^d$ if it is congruent to every other bar-and-joint framework 
in $\mathbb{R}^d$ with the same edge lengths.

The problem of whether a given framework (or all frameworks
from a given family) is globally rigid was raised both 
in mathematics and in its applications.
As a purely mathematical problem it was raised, loosely speaking, 
in distance geometry (see, e.\,g., \cite{BS14,Co05,CW10,GHT10}), 
graph theory (see, e.\,g., \cite{He92,SS10,Ta15}),
matroid theory (see, e.\,g., \cite{JJ05})), etc.
Since frameworks are a natural model for real-world mechanisms 
and molecules, this problem appeared also in classical mechanics of
mechanisms (see, e.\,g., \cite{Ka10,St10}), 
the mechanics of microporous materials (see, e.\,g., \cite{TDKR14}), 
stereochemistry (see, e.\,g., \cite{He95}), 
molecular biology (see, e.\,g., \cite{AMG13,SSW14}), etc.

A special reason for our interest in globally rigid frameworks
is that, for infinitesimally rigid frameworks,
which may be treated as an infinitesimal analogue of
globally rigid frameworks, a classical theorem reads that the property
of a framework `to be infinitesimally rigid' is invariant under
projective transformations (see, e.\,g., \cite{Iz09,We84}).
In contrast, Theorem \ref{thrm1}, the main result of this note, reads that 
the property of a framework `to be globally rigid' is not invariant 
even under affine transformations.

We recall basic definitions and notation from geometric global rigidity 
of bar-and-joint frameworks (see, e.\,g., \cite{CW10}).

Let $d\geqslant 1$ be an integer and $G=(V,E)$ be a graph.
Here $V$ and $E$ are the sets of vertices and edges of $G$, respectively.

A {\it bar-and-joint framework\,} $G(\mathbf{p})$ 
in $\mathbb{R}^d$ (or a {\it framework}, for short) is a graph $G$ 
and a {\it configuration} $\mathbf{p}$ which assigns a point 
$\mathbf{p}_i\in\mathbb{R}^d$ to each vertex $i\in V$.
For each $i\in V$, $\mathbf{p}_i$ is called a {\it joint} 
of $G(\mathbf{p})$ and, for each edge $\{i, j\}\in E$, 
the staight-line segment with the endpoints $\mathbf{p}_i$ 
and $\mathbf{p}_j$ is called a {\it bar}.

Note that some authors use the term {\it straight-line realization
of a graph $G$} or {\it geometric realization
of a graph $G$} instead of the term bar-and-joint framework 
$G(\mathbf{p})$ (see, e.\,g., \cite{He92,JJ05}). 

We denote the Euclidean distance between 
$\mathbf{x}, \mathbf{y} \in\mathbb{R}^d$
by $|\mathbf{x}- \mathbf{y}|$.

Two frameworks $G(\mathbf{p})$ and $G(\mathbf{q})$
are {\it equivalent} to each other if 
$|\mathbf{p}_i-\mathbf{p}_j|=|\mathbf{q}_i-\mathbf{q}_j|$
for every  $\{i, j\}\in E$ and
are {\it congruent} to each other if 
$|\mathbf{p}_i-\mathbf{p}_j|=|\mathbf{q}_i-\mathbf{q}_j|$
for all  $i, j\in V$.

A framework $G(\mathbf{p})$ is {\it globally rigid in} 
$\mathbb{R}^d$ if all frameworks $G(\mathbf{q})$ in 
$\mathbb{R}^d$ which are equivalent to $G(\mathbf{p})$ 
are congruent to $G(\mathbf{p})$.

Let $A:\mathbb{R}^d\to\mathbb{R}^d$ be an affine transformation
and $G(\mathbf{p})$ be a framework in $\mathbb{R}^d$.
We write $A\mathbf{p}$ for a configuration of the graph $G$
given by the formulas $(A\mathbf{p})_k=A(\mathbf{p}_k)$, $k\in V$.
In the sequel, where this cannot cause misunderstanding, we 
write $A\mathbf{x}$ instead of $A(\mathbf{x})$
for $\mathbf{x}\in \mathbb{R}^d$.

The main result of this note is the following 

\begin{thrm}\label{thrm1}
For every $d\geqslant 2$, there is a framework $G_d(\mathbf{p})$
in $\mathbb{R}^d$ and an affine transformation
$A:\mathbb{R}^d\to\mathbb{R}^d$ such that
$G_d(\mathbf{p})$ is not globally rigid in $\mathbb{R}^d$ while 
$G_d(A\mathbf{p})$ is globally rigid in $\mathbb{R}^d$.
\end{thrm}

Note that, Theorem \ref{thrm1} is not new for the case $d=2$.
In fact, in \cite[Example 8.3]{CW10}, a configuration in $\mathbb{R}^2$ is constructed 
that is not globally rigid in $\mathbb{R}^2$, while its affine image is globally 
rigid in $\mathbb{R}^2$.
That example relies on the properties of coning and stress matrices, specific techniques 
for the study of global rigidity of bar-and-joint frameworks, developed in \cite{CW10}
and articles mentioned there.

In contrast, the proof of Theorem \ref{thrm1} given in Section \ref{s3} below
is valid in Euclidean spaces of all dimensions and is straightforward,
in particular, it does not involve coning.

\section{Preliminary considerations}\label{s2}

Let $G^*=(V^*, E^*)$ denote the graph such that
$V^*=\{1, 2, 3, 4, 5\}$ 
and
$E^*=\bigl\{\{1,2\}, \{1,3\}, \{1,4\}, \{1,5\}, \{2,3\}, \{2,4\}, \{3,5\}, \{4,5\}\bigr\}$.
$G^*$ is shown in Fig.~\ref{fig1}(a). 
Let $\mathbf{p}^*$ denote the configuration of  
$G^*$ in $\mathbb{R}^2$ which is given by the following formulas:
$\mathbf{p}^*_1=(0, 2)$,
$\mathbf{p}^*_2=(-1/4, 1/2)$,
$\mathbf{p}^*_3=(21/20, 9/10)$,
$\mathbf{p}^*_4=(-1, 0)$,
and
$\mathbf{p}^*_5=(1, 0)$.
$G(\mathbf{p}^*)$ is shown in Fig.~\ref{fig1}(b).

Let $A:\mathbb{R}^2\to\mathbb{R}^2$ 
be the affine transformation given by the formula
$A(x_1,x_2)=(x_1,{x_2}/2)$.

\begin{lemma}\label{l1}
Let the graph $G^*$, the framework $G^*(\mathbf{p}^*)$ and 
the affine transformation $A$ be as above in this Section. 
Then $G^*(\mathbf{p}^*)$ is not globally rigid in $\mathbb{R}^2$, 
while $G^*(A\mathbf{p}^*)$ is globally rigid in $\mathbb{R}^2$.
\end{lemma}

\begin{proof}
Let $\mathbf{q}^*$ be the configuration of $G^*$ in $\mathbb{R}^2$
which is given by the following formulas:
$\mathbf{q}^*_1=\mathbf{p}^*_1=(0, 2)$,
$\mathbf{q}^*_2=(-21/20, 9/10)$,
$\mathbf{q}^*_3=(1/4, 1/2)$,
$\mathbf{q}^*_4=\mathbf{p}^*_4=(-1, 0)$,
and
$\mathbf{q}^*_5=\mathbf{p}^*_5=(1,0)$.
$G^*(\mathbf{q}^*)$ is shown in Fig.~\ref{fig1}(c).
\begin{figure}
\begin{center}
\includegraphics[width=0.9\textwidth]{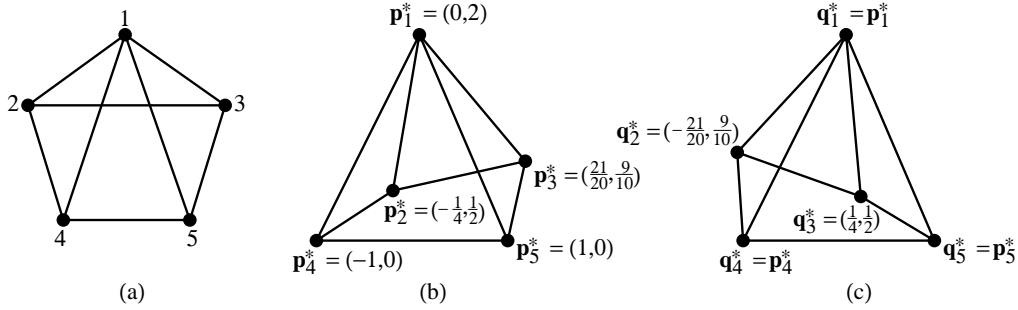}
\end{center}
\caption{(a): Graph $G^*$. (b): Framework $G^*(\mathbf{p}^*)$. (c): Framework
$G^*(\mathbf{q}^*)$.}\label{fig1}
\end{figure}

We can say that $G^*(\mathbf{q}^*)$ 
is obtained from $G^*(\mathbf{p}^*)$ by replacing  
the joint $\mathbf{p}^*_2$ by the joint $\mathbf{q}^*_2$ 
and by replacing the joint $\mathbf{p}^*_3$ by the joint 
$\mathbf{q}^*_3$.
Note that $\mathbf{q}^*_2$ is chosen in such a way
that it is symmetrical to $\mathbf{p}^*_2$ 
with respect to the line $\mathbf{p}^*_1\mathbf{p}^*_4$ 
(here and subsequently $\mathbf{x}\mathbf{y}$ denotes 
the straight line passing through the points
$\mathbf{x},\mathbf{y}\in \mathbb{R}^2$).
Similarly,  $\mathbf{q}^*_3$ is chosen in such a way
that it is symmetrical to $\mathbf{p}^*_3$ 
with respect to the line $\mathbf{p}^*_1\mathbf{p}^*_5$.

Using the symmetry of some parts of the framework
$G^*(\mathbf{p}^*)$, it is easy to see that 
$|\mathbf{q}^*_2-\mathbf{q}^*_3| = |\mathbf{p}^*_2-\mathbf{p}^*_3|$.
Hence, the frameworks  
$G^*(\mathbf{p}^*)$ 
and
$G^*(\mathbf{q}^*)$ 
are equivalent to each other. 
However, we can arrive at the same conclusion by direct calculating
the lengths of all bars of these frameworks.

On the other hand, $G^*(\mathbf{p}^*)$ 
and $G^*(\mathbf{q}^*)$ 
are not congruent to each other since
$|\mathbf{p}^*_2-\mathbf{p}^*_5| \neq |\mathbf{q}^*_2-\mathbf{q}^*_5|$.
In fact, direct calculations show that
$|\mathbf{p}^*_2-\mathbf{p}^*_5| = \sqrt{29}/4$ 
and
$|\mathbf{q}^*_2-\mathbf{q}^*_5| = \sqrt{2005}/20$.

Hence, $G^*(\mathbf{p}^*)$ is globally rigid in $\mathbb{R}^2$.

Let $\widetilde{G}^*=(\widetilde{V}^*, \widetilde{E}^*)$ 
denote the graph such that
$\widetilde{V}^*=V^*$ 
and
$\widetilde{E}^*=E^*\setminus\bigl\{\{2, 3\}\bigr\}$.
$\widetilde{G}^*$ is shown in Fig.~\ref{fig2}(a).
\begin{figure}
\begin{center}
\includegraphics[width=0.9\textwidth]{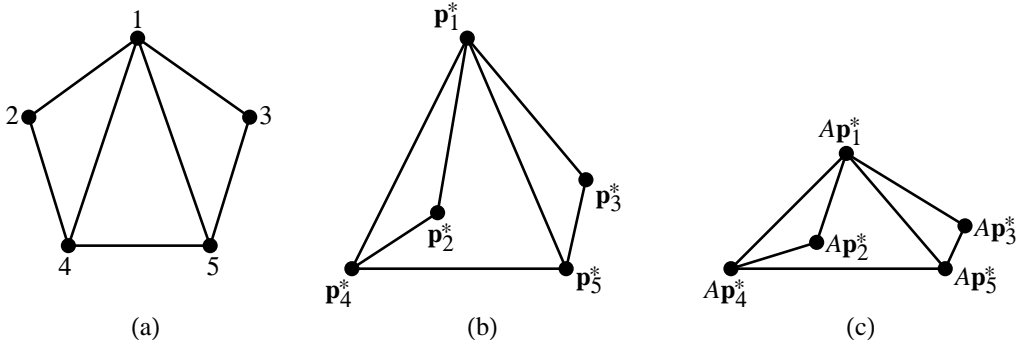}
\end{center}
\caption{(a): Graph $\widetilde{G}^*$. (b): Framework $\widetilde{G}^*(\mathbf{p}^*)$. (c): Framework
$\widetilde{G}^*(A\mathbf{p}^*)$.}\label{fig2}
\end{figure}
The configuration $\mathbf{p}^*$ of the graph $G^*$
and the affine transformation $A:\mathbb{R}^2\to\mathbb{R}^2$
constructed above in this Section define also the frameworks 
$\widetilde{G}^*(\mathbf{p}^*)$ 
and
$\widetilde{G}^*(A\mathbf{p}^*)$,
which are shown in Fig.~\ref{fig2}(b) and Fig.~\ref{fig2}(c),
respectively.

Every framework in $\mathbb{R}^2$, which is equivalent to
$\widetilde{G}^*(A\mathbf{p}^*)$, is congruent to one of the frameworks
shown in Fig.~\ref{fig3}, where
$\mathbf{x}_2=(-3/4, 3/4)$ 
and
$\mathbf{x}_3=(11/20, -1/20)$.
\begin{figure}
\begin{center}
\includegraphics[width=1\textwidth]{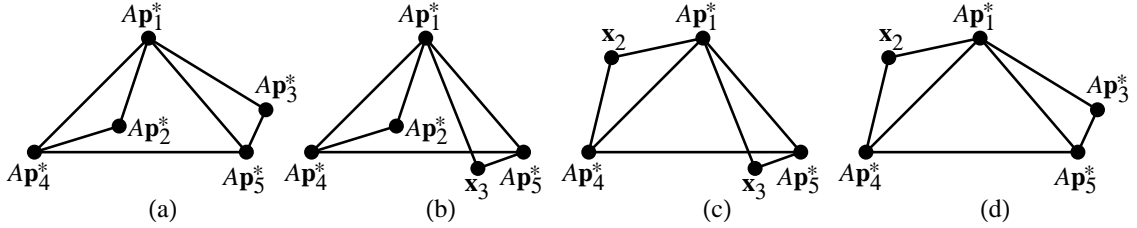}
\end{center}
\caption{Any framework equivalent to $\widetilde{G}^*(A\mathbf{p}^*)$ is congruent to one of
the frameworks shown in (a)--(d).}\label{fig3}
\end{figure}
Note that the points $\mathbf{x}_2$ and $\mathbf{x}_3$ 
are chosen in such a way that $\mathbf{x}_2$ and $A\mathbf{p}_2^*$ 
are symmetrical to each other with respect to the line 
$A\mathbf{p}_1^*A\mathbf{p}_4^*$, and
$\mathbf{x}_3$ and $A\mathbf{p}_3^*$
are symmetrical to each other with respect to the line
$A\mathbf{p}_1^*A\mathbf{p}_5^*$.

Direct calculations show that
$|A\mathbf{p}_2^*-A\mathbf{p}_3^*|=\sqrt{173}/10$
(see Fig.~\ref{fig3}(a)),
$|A\mathbf{p}_2^*-\mathbf{x}_3|=\sqrt{73}/10$
(see Fig.~\ref{fig3}(b)),
$|\mathbf{x}_2-\mathbf{x}_3|=\sqrt{233}/10$
(see Fig.~\ref{fig3}(c)),
and
$|\mathbf{x}_2-A\mathbf{p}_3^*|=3\sqrt{37}/10$
(see Fig.~\ref{fig3}(d)).
Since among these numbers there are no two equal, 
the framework $G^*(A\mathbf{p}^*)$ is globally rigid in $\mathbb{R}^2$.
For the convenience of the reader $G^*(A\mathbf{p}^*)$ 
is shown in Fig.~\ref{fig4}(a).
\end{proof}

\section{Proof of Theorem~\ref{thrm1}}\label{s3}

In the proof of Theorem~\ref{thrm1} given below, we mainly use
the same ideas that were used in Section~\ref{s2} in the proof of Lemma~\ref{l1}.
The novelty is in details that are needed for transition from the plane to space.

\begin{proof}
Let $\widetilde{G}_d=(\widetilde{V}_d,\widetilde{E}_d)$
denote the graph such that $\widetilde{V}_d=\{1,2,3,4,\dots, 2^{d-1}+2, 2^{d-1}+3\}$
and an unordered pair of non-coincident vertices 
$i,j$ is an edge of $\widetilde{G}_d$ (i.\,e. $\{i,j\}\in \widetilde{E}_d$)
if and only if one of the conditions (i)--(iii) is fulfilled:
 
\begin{figure}
\begin{center}
\includegraphics[width=0.6\textwidth]{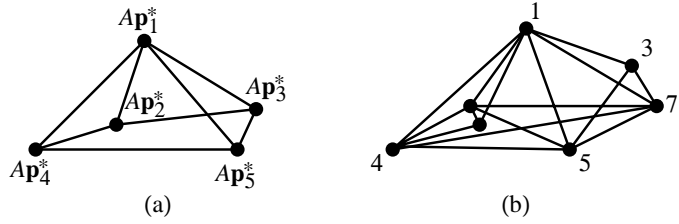}
\end{center}
\caption{(a): Globally rigid framework
$G^*(A\mathbf{p}^*)=G_2(\mathbf{p}^*)$. (b): Graph
$G_3$ (vertices $2$ and $6$ are left unmarked intentionally).}\label{fig4}
\end{figure}

(i) $i,j\in \{1, 4, 5, \dots, 2^{d-1}+2, 2^{d-1}+3\}$ 
(i.\,e. the complete graph with the vertices 
$1, 4, 5, \dots, 2^{d-1}+2, 2^{d-1}+3$ 
is contained in $\widetilde{G}_d$);

(ii) the unordered pair $\{i,j\}$ coincides with either $\{1, 2\}$
or $\{2, 2k\}$, where $2\leqslant k \leqslant 2^{d-2}+1$
(i.\,e. the vertex 2 is connected by an edge to each of the vertices
$1, 4, 6, \dots, 2k, \dots, 2^{d-1}+2$);

(iii) the unordered pair $\{i,j\}$ coincides with either $\{1, 3\}$, 
or $\{3, 2k+1\}$, where $2\leqslant k \leqslant 2^{d-2}+1$
(i.\,e. the vertex 3 is connected by an edge to each of the vertices
$1, 5, 7, \dots, 2k+1, \dots, 2^{d-1}+3$).

Let $G_d=(V_d,E_d)$ denote the graph such that
$V_d=\widetilde{V}_d$ 
and
$E_d=\widetilde{E}_d\cup \bigl\{\{2,3\}\bigr\}$.
$G_2$ is shown in Fig.~\ref{fig1}(a), 
$\widetilde{G}_2$ is shown in Fig.~\ref{fig2}(a),
and $G_3$ is shown in Fig.~\ref{fig4}(b).

The points $(\pm 1,\pm 1, \dots, \pm 1)\in \mathbb{R}^{d-2}$
are the vertices of a cube with edge length 2.
We enumerate them in an arbitrary way using the numbers
$1\leqslant k\leqslant 2^{d-2}$ and denote them by $w_k$.
So, $w_k=(\pm 1,\pm 1, \dots, \pm 1)$ 
with the proper selection of plus and minus signs.

Let $\mathbf{p}$ denote the configuration of the graphs 
$G_d$ and $\widetilde{G}_d$ in $\mathbb{R}^d$ that 
is given by the following formulas:
$\mathbf{p}_1=(0, 2, 0, \dots, 0)$, 
$\mathbf{p}_2=(-1/4, 1/2, 0, \dots, 0)$, 
$\mathbf{p}_3=(21/20, 9/10, 0, \dots, 0)$, 
$\mathbf{p}_{2j}=(-1, 0, w_{j-1})$, 
and
$\mathbf{p}_{2j+1}=(1,0,w_{j-1})$,
where $2\leqslant j\leqslant 2^{d-2}+1$.

Let $\mathbf{q}$ denote the configuration of the graphs 
$G_d$ and $\widetilde{G}_d$ in $\mathbb{R}^d$ that 
is given by the following formulas: 
$\mathbf{q}_1=\mathbf{p}_1$, 
$\mathbf{q}_2=(-21/20, 9/10, 0, \dots, 0)$, 
$\mathbf{q}_3=(1/4, 1/2, 0, \dots, 0)$, 
and
$\mathbf{q}_k=\mathbf{p}_k$,
where $k=4,5,\dots, 2^{d-1}+2, 2^{d-1}+3$.

Note that $\mathbf{q}_2$ is chosen in such a way
that it is symmetrical to $\mathbf{p}_2$
with respect to the hyperplane in $\mathbb{R}^d$
passing through the points
$\mathbf{p}_1$, $\mathbf{p}_4$, $\mathbf{p}_6$, \dots ,
$\mathbf{p}_{2^{d-1}}$, $\mathbf{p}_{2^{d-1}+2}$. 
Similarly, $\mathbf{q}_3$ is chosen in such a way
that it is symmetrical to $\mathbf{p}_3$
with respect to the hyperplane in $\mathbb{R}^d$
passing through the points 
$\mathbf{p}_1$, $\mathbf{p}_5$, $\mathbf{p}_7$, \dots ,
$\mathbf{p}_{2^{d-1}+1}$, $\mathbf{p}_{2^{d-1}+3}$.

As in the proof of Lemma~\ref{l1}, a direct verification
shows that the frameworks $G_d(\mathbf{p})$ and $G_d(\mathbf{q})$
are equivalent to each other, but not congruent.
Hence, $G_d(\mathbf{p})$ is not globally rigid.

Let $A:\mathbb{R}^d\to\mathbb{R}^d$
be the affine transformation given by the formula
$A(x_1, x_2, x_3, \dots, x_d) = (x_1, {x_2}/2, x_3, \dots, x_d)$
(i.\,e. $A$ is the contraction with the factor of 2 along 
the second axis of $\mathbb{R}^d$).

Let $\mathbf{r}$ denote the configuration of the graphs 
$G_d$ and $\widetilde{G}_d$ in $\mathbb{R}^d$ that 
is given by the following formulas: 
$\mathbf{r}_1=A\mathbf{p}_1$, 
$\mathbf{r}_2=A\mathbf{p}_2$, 
$\mathbf{r}_3=(11/2, -1/20, 0, \dots, 0)$, 
and
$\mathbf{r}_k=A\mathbf{p}_k$,
where $k=4,5,\dots, 2^{d-1}+2, 2^{d-1}+3$.
Note that $\mathbf{r}_3$ is symmetrical to 
$A\mathbf{p}_3$ with respect to the hyperplane in $\mathbb{R}^d$
passing through the points
$A\mathbf{p}_1$, $A\mathbf{p}_5$, $A\mathbf{p}_7$, \dots ,
$A\mathbf{p}_{2^{d-1}+1}$, $A\mathbf{p}_{2^{d-1}+3}$.

Let $\mathbf{s}$ denote the configuration of the graphs 
$G_d$ and $\widetilde{G}_d$ in $\mathbb{R}^d$ that 
is given by the following formulas:  
$\mathbf{s}_1=A\mathbf{p}_1$, 
$\mathbf{s}_2=(-3/4, 3/4, 0, \dots, 0)$, 
$\mathbf{s}_3=\mathbf{r}_3$, 
and
$\mathbf{s}_k=A\mathbf{p}_k$,
where $k=4,5,\dots, 2^{d-1}+2, 2^{d-1}+3$.
Note that $\mathbf{s}_2$ is symmetrical to 
$A\mathbf{p}_2$  with respect to the hyperplane in $\mathbb{R}^d$,
passing through the points
$A\mathbf{p}_1$, $A\mathbf{p}_4$, $A\mathbf{p}_6$, \dots ,
$A\mathbf{p}_{2^{d-1}}$, $A\mathbf{p}_{2^{d-1}+2}$.

Finally, let $\mathbf{t}$ denote the configuration of the graphs 
$G_d$ and $\widetilde{G}_d$ in $\mathbb{R}^d$ that 
is given by the following formulas:  
$\mathbf{t}_1=A\mathbf{p}_1$, 
$\mathbf{t}_2=\mathbf{s}_2$, 
$\mathbf{t}_3=\mathbf{r}_3$, 
and
$\mathbf{t}_k=A\mathbf{p}_k$,
where $k=4,5,\dots, 2^{d-1}+2, 2^{d-1}+3$.

As in the proof of Lemma~\ref{l1}, 
we make sure that any framework equivalent to 
$\widetilde{G}_d(A\mathbf{p})$ in $\mathbb{R}^d$ 
is congruent to one of the frameworks 
$\widetilde{G}_d(A\mathbf{p})$,
$\widetilde{G}_d(\mathbf{r})$, 
$\widetilde{G}_d(\mathbf{s})$, or 
$\widetilde{G}_d(\mathbf{t})$.
Direct calculations show that
$|A\mathbf{p}_2-A\mathbf{p}_3|=\sqrt{173}/10$,
$|\mathbf{r}_2-\mathbf{r}_3|=\sqrt{73}/10$,
$|\mathbf{s}_2-\mathbf{s}_3|=\sqrt{233}/10$,
and
$|\mathbf{t}_2-\mathbf{t}_3|=3\sqrt{37}/10$.
Since among these numbers there are no two equal,
none of the frameworks
$G_d(A\mathbf{p})$,
$G_d(\mathbf{r})$,
$G_d(\mathbf{s})$,
and
$G_d(\mathbf{t})$
are equivalent to each other.
Consequently, 
$G_d(A\mathbf{p})$
is globally rigid.
\end{proof}

\section{Acknowledgements}\label{s4}

The author is indebted to Professor Steven Gortler for bringing his attention to 
Example 8.3 in \cite{CW10}.

\bibliographystyle{plain}
\bibliography{alex_bib}
\end{document}